\documentclass[a4paper,10pt,reqno, english]{amsart}  
\usepackage[utf8]{inputenc}
\usepackage[T1]{fontenc}
\usepackage{amsmath,amsthm}
\usepackage{amsfonts,amssymb,enumerate}
\usepackage{url,paralist}
\usepackage{mathtools}  
\usepackage[colorlinks=true,urlcolor=blue,linkcolor=red,citecolor=magenta]{hyperref}
\usepackage{enumerate}
\usepackage{anysize}

\theoremstyle{plain}
\newtheorem{theorem}{Theorem}[section]
\newtheorem{lemma}[theorem]{Lemma}
\newtheorem{corollary}[theorem]{Corollary}

\newtheorem{conjecture}[theorem]{Conjecture}
\newtheorem*{theorem*}{Theorem}

\theoremstyle{definition}

\newtheorem{example}[theorem]{Example}
\newtheorem{remark}[theorem]{Remark}

\newcommand{\R}{\mathbb{R}}
\newcommand{\Z}{\mathbb{Z}}
\DeclareMathOperator{\conv}{\mathrm{conv}}
\DeclareMathOperator{\aff}{\mathrm{aff}}
\newcommand{\KG}{\mathrm{KG}}

\newcommand{\Sym}{\mathfrak{S}}

\begin{document}

\title{Intersection patterns of finite sets and of convex sets}



\author{Florian Frick}
\address{Dept.\ Math., Cornell University, Ithaca, NY 14853, USA}
\email{ff238@cornell.edu}

\date{\today}
\maketitle


\begin{abstract}
\small The main result is a common generalization of results on lower bounds for the chromatic number 
of $r$-uniform hypergraphs and some of the major theorems in Tverberg-type theory,
which is concerned with the intersection pattern of faces in a simplicial complex when continuously
mapped to Euclidean space. 
As an application we get a simple proof of a generalization of a result of Kriz
for certain parameters. This specializes to a short and simple proof of Kneser's conjecture. Moreover, combining
this result with recent work of Mabillard and Wagner we show that the existence of certain equivariant 
maps yields lower bounds for chromatic numbers. We obtain an essentially elementary proof of the result
of Schrijver on the chromatic number of stable Kneser graphs. In fact, we show that every neighborly even-dimensional
polytope yields a small induced subgraph of the Kneser graph of the same chromatic number. 
We furthermore use this geometric viewpoint to give tight lower bounds for the chromatic number of
certain small subhypergraphs of Kneser hypergraphs.
\end{abstract}

\section{Introduction}

\noindent
Kneser conjectured~\cite{kneser1955} that for any partition of the $n$-element subsets of $\{1, 2, \dots, 2n+k\}$ into
$k+1$ classes there exists one class that contains two disjoint sets. This can be easily reformulated 
into a question about the chromatic number of a graph: the vertices of the \emph{Kneser graph} 
$\KG(n,2n+k)$ correspond to the $n$-element subsets of $\{1, \dots, 2n+k\}$ with an edge between
two vertices if the corresponding sets are disjoint. Kneser's conjecture then states that $\chi(\KG(n,2n+k)) \ge k+2$.
A simple greedy coloring shows $\chi(\KG(n,2n+k)) \le k+2$. This conjecture was proved by Lov\'asz~\cite{lovasz1978}
using the Borsuk--Ulam theorem in one of the earliest applications of algebraic topology to a 
combinatorial problem. 

More generally, one could ask for sufficient conditions on a finite system $G$ of
finite sets such that in any partition of $G$ into $k$ classes there is one class with $r$ pairwise disjoint sets.
This can be reformulated into a question about lower bounds for the chromatic number of $r$-uniform
hypergraphs. A rather general lower bound is due to Kriz~\cite{kriz1992, kriz2000c}. While his proof is topological 
-- using equivariant cohomology -- the condition on $G$ that Kriz derives is purely combinatorial. The approach in Section~\ref{sec:hypergraphs}
yields an elementary proof of Kriz's result (for certain parameters) and extends it by also taking the topology
of $G$ into account, more precisely the topology of the simplicial complex of all sets with no subset in~$G$.

Tverberg-type theory is a geometric analog of the intersection patterns of finite sets, where now instead of finite sets and their intersections
one is given a finite set of points in Euclidean space and is interested in which restrictions prohibit convex
hulls of $r$ pairwise disjoint subsets to have a common point of intersection. Usually this is formulated as the existence of a point of
$r$-fold incidence among pairwise disjoint faces of a simplicial complex $K$ when $K$ is affinely mapped to~$\R^d$. Continuous analogs of this
theory have turned out to be of major interest. We will summarize the main results in Tverberg-type theory in Section~\ref{sec:tverberg}.

An idea already present in papers of Sarkaria~\cite{sarkaria1990, sarkaria1991} is to relate Tverberg-type theory and colorings of hypergraphs to one another.
For a simplicial complex $L$ and a subcomplex $K \subseteq L$ denote by $\KG^r(K,L)$ the $r$-uniform hypergraph with
vertices corresponding to the inclusion-minimal faces of $L$ that are not contained in~$K$ and a hyperedge spanned by 
$r$ vertices if the corresponding faces of $L$ are pairwise disjoint. We use the constraint method of Blagojevi\'c, Ziegler,
and the author~\cite{blagojevic2014} to give a simple and elementary proof of the following result relating the combinatorics of
missing faces of a simplicial complex to Tverberg-type intersection results; see Theorem~\ref{thm:constraints}.

\begin{theorem*}
	Let $d, k \ge 0$ and $r \ge 2$ be integers, $K \subseteq L$ simplicial complexes such that
	for every continuous map $F\colon L \longrightarrow \R^{d+k}$ there are $r$ pairwise disjoint
	faces $\sigma_1, \dots, \sigma_r$ of $L$ such that $F(\sigma_1) \cap \dots \cap F(\sigma_r) \ne \emptyset$.
	Suppose $\chi(\KG^r(K, L)) \le k$. Then for every continuous map
	$f\colon K \longrightarrow \R^d$ there are $r$ pairwise disjoint faces $\sigma_1, \dots, \sigma_r$
	of $K$ such that $f(\sigma_1) \cap \dots \cap f(\sigma_r) \ne \emptyset$.
\end{theorem*}

Given a simplicial complex $L$ such that for every continuous map $F\colon L \longrightarrow \R^{d+k}$
there are $r$ pairwise disjoint faces $\sigma_1, \dots, \sigma_r$ of $L$ such that $F(\sigma_1) \cap \dots \cap F(\sigma_r) \ne \emptyset$,
this theorem can be used in two directions: By establishing the upper bound $\chi(\KG^r(K, L)) \le k$ one can deduce
an $r$-fold intersection result for continuous maps $f\colon K \longrightarrow \R^d$, whereas by exhibiting a continuous 
map $f\colon K \longrightarrow \R^d$ without such an $r$-fold intersection one can deduce the lower bound
$\chi(\KG^r(K, L)) \ge k+1$. Thus this relates intersection patterns of continuous images of faces in a simplicial complex to
intersection patterns of finite sets in the guise of chromatic numbers of intersection hypergraphs.

The theorem has the following results as simple corollaries:
\begin{compactitem}
	\item the generalized van Kampen--Flores theorem of Sarkaria~\cite{sarkaria1991} and Volovikov~\cite{volovikov1996},
	\item the colored Tverberg theorems of type A of \v Zivaljevi\'c and Vre\'cica~\cite{zivaljevic1992} 
		as well as the generalization of type B due to Vre\'cica and \v Zivaljevi\'c~\cite{vrecica1994} (see Corollary~\ref{cor:constraints} for 
		a common generalization of these colored Tverberg theorems and the generalized van Kampen--Flores theorem),
	\item Kneser's conjecture proven by Lov\'asz~\cite{lovasz1978}, see Theorem~\ref{thm:kneser}, and Dol'nikov's theorem~\cite{dolnikov1988},
	\item more generally, Kriz's lower bound for the chromatic number of $r$-uniform intersection hypergraphs~\cite{kriz1992} for certain parameters, 
		see Corollary~\ref{cor:lower-bounds},
	\item Schrijver's theorem on stable Kneser graphs~\cite{schrijver1978}, see Section~\ref{sec:schrijver}.
\end{compactitem}

The Tverberg-type result follow from combining the topological Tverberg theorem, see Theorem~\ref{thm:top-tverberg},
with greedy colorings of hypergraphs. The lower bounds for the chromatic number follow from a codimension count
for general position affine map or from understanding the geometry of points in cyclic position. 

Special cases of this result are already implicit in Sarkaria's papers~\cite{sarkaria1990, sarkaria1991} and more 
explicit in Matou\v sek's book~\cite[Theorem~6.7.3]{matousek2008}, where the special case that $r$ is
a prime and $L$ is the simplex on the vertex set of $K$ is proven. However, the proof presented here is significantly simpler
and does not need to appeal to $\Z/r$-indices or other methods from algebraic topology. It is a combination of Sarkaria's 
coloring ideas with the constraint method of~\cite{blagojevic2014}.

While the main focus of this paper is to significantly simplify proofs of known results and relate intersection patterns of convex
sets and of finite sets to one another, we can also use the main theorem to further extend the known results. 
We derive the following new results:
\begin{compactitem}
	\item The consequences above follow from combining the main result and the topological Tverberg theorem. We get
		proper extension for $r$ a prime by using the optimal colored Tverberg theorem of Blagojevi\'c, Matschke,
		and Ziegler~\cite{blagojevic2009} instead; see Corollary~\ref{cor:optimal-colored}.
	\item Missing faces of neighborly even-dimensional polytopes induce subgraphs of the Kneser graph $\KG(n,2n+k)$ that still have chromatic number $k+2$,
		where Schrijver's theorem is the special case of cyclic polytopes. We thus obtain many subgraphs of $\KG(n,2n+k)$ on 
		few vertices and with chromatic number~$k+2$; see Corollary~\ref{cor:neighborly-kneser}.
	\item More generally, we show that if $K$ triangulates $S^{d-1}$ on $n$ vertices, then the graph with vertex set
		the missing faces of $K$ and an edge for each pair of disjoint faces has chromatic number~${n-d}$; see Theorem~\ref{thm:spherical-kneser}.
	\item We show that the chromatic number of Kneser hypergraphs does not decrease if one restricts to
		$k$-element sets that are $(\frac{r(k-3)}{2(k-1)}+1)$-stable on average; see Theorem~\ref{thm:avg}.
	\item We remark that lower bounds for chromatic numbers of Kneser hypergraphs can be established by 
		exhibiting equivariant maps --- and not as usual by showing the nonexistence of an equivariant map,
		which often is more difficult. This follows from combining the main theorem with recent work of 
		Mabillard and Wagner~\cite{mabillard2016}; see Corollary~\ref{cor:metastable}.
\end{compactitem}

\section*{Acknowledgements}

\noindent
I am grateful to G\"unter M. Ziegler for very good comments and suggestions.

\section{Tverberg-type theorems}
\label{sec:tverberg}

\noindent
Here we collect some of the main results from Tverberg-type theory and refer to Matou\v sek's book~\cite{matousek2008}
for details. We denote the $N$-dimensional 
simplex by~$\Delta_N$. A classical lemma of Radon~\cite{radon1921} states that any $d+2$ points in $\R^d$ can be split 
into two sets with intersecting convex hulls. Equivalently, this can be phrased as: any affine map
$f\colon \Delta_{d+1} \longrightarrow \R^d$ identifies points from two disjoint faces of~$\Delta_{d+1}$.
The following theorem states that this remains true if one replaces affine by continuous.

\begin{theorem}[Topological Radon theorem, Bajm\'oczy and B\'ar\'any~\cite{bajmoczy1979}]
	For any continuous map \linebreak${f\colon \Delta_{d+1} \longrightarrow \R^d}$ there are 
	two disjoint faces $\sigma_1$ and $\sigma_2$ of $\Delta_{d+1}$ such that
	$f(\sigma_1) \cap f(\sigma_2) \ne \emptyset$.
\end{theorem}

This theorem follows from the Borsuk--Ulam theorem stating that any continuous map
$S^d \longrightarrow \R^d$ identifies two antipodal points. We will use the topological 
Radon theorem to establish lower bounds on the chromatic number of Kneser graphs. 

To obtain lower bounds for the chromatic number of $r$-uniform hypergraphs, we
need a generalization of the topological Radon theorem to multiple intersections. In the case
of an affine map such a theorem is due to Tverberg~\cite{tverberg1966}. He showed that for $N=(r-1)(d+1)$ any
affine map $f\colon \Delta_N \longrightarrow \R^d$ identifies points from $r$ pairwise disjoint
faces. Extending this result to continuous maps turned out to be a major problem. This was
accomplished for $r$ a prime by B\'ar\'any, Shlosman, and Sz\H ucz~\cite{barany1981} and more generally for
$r$ a power of a prime by \"Ozaydin~\cite{ozaydin1987}.

\begin{theorem}[Topological Tverberg theorem, B\'ar\'any, Shlosman, and Sz\H ucz~\cite{barany1981}, \"Ozaydin~\cite{ozaydin1987}]
\label{thm:top-tverberg}
	Let $r \ge 2$ be a prime power, $d \ge 0$ an integer, and $N = (r-1)(d+1)$. Then 
	for any continuous map $f\colon\Delta_N \longrightarrow \R^d$ there are $r$ pairwise
	disjoint faces $\sigma_1, \dots, \sigma_r$ of $\Delta_N$ such that
	$f(\sigma_1) \cap \dots \cap f(\sigma_r) \ne \emptyset$.
\end{theorem}

\"Ozaydin already showed that the obstruction used to prove the topological Tverberg theorem
for $r$ a prime power vanishes for all other~$r$. Counterexamples to the topological Tverberg
conjecture, that is, to the statement that Theorem~\ref{thm:top-tverberg} hold for all~$r$, came 
about when Mabillard and Wagner~\cite{mabillard2015} proved that for simplicial complexes of dimension at most $d-3$ 
(and other technical conditions) the vanishing of the $r$-fold van Kampen obstruction implies 
the existence of a continuous map that does not identify points from $r$ pairwise disjoint faces.
This together with~\cite{blagojevic2014} -- the reduction needed here was also sketched earlier by Gromov~{\cite[p.~445f.]{gromov2010}} -- yields 
counterexamples to the topological Tverberg conjecture for any $r$ that is not a prime power as 
pointed out in~{\cite{frick2015, blagojevic2015}}. The smallest counterexample is due to Avvakumov, Mabillard, Skopenkov, and
Wagner~\cite{avvakumov2015}. By $N(r,d)$ we denote the smallest integer $N$ such that the statement of 
Theorem~\ref{thm:top-tverberg} holds for parameters $N, r$, and~$d$. For lower and upper 
bounds on $N(r,d)$ for $r$ not a prime power see~\cite{blagojevic2015}. 

Three important variations of Theorem~\ref{thm:top-tverberg} are the following. Here we denote
by $K^{(k)}$ the \emph{$k$-skeleton} of the simplicial complex~$K$, and $K*L$ denotes the 
\emph{join} of the complexes $K$ and~$L$. See~\cite{matousek2008} for more details.

\begin{theorem}[Generalized van Kampen--Flores theorem, Sarkaria~\cite{sarkaria1991}, Volovikov~\cite{volovikov1996}]
\label{thm:vanKampen}
	Let $r \ge 2$ be a prime power, $d \ge 0$ an integer, $N = (r-1)(d+2)$, and $k$ an integer 
	with $r(k+2) > N+1$. Then for any continuous map $f\colon\Delta_N^{(k)} \longrightarrow \R^d$ 
	there are $r$ pairwise disjoint faces $\sigma_1, \dots, \sigma_r$ of $\Delta_N^{(k)}$ such that
	$f(\sigma_1) \cap \dots \cap f(\sigma_r) \ne \emptyset$.
\end{theorem}

\begin{theorem}[Colored Tverberg theorem, Vre\'cica and \v Zivaljevi\'c~\cite{vrecica1994}]
\label{thm:colored-tverberg}
	Let $r \ge 2$ be a prime power, $d \ge 0$ an integer, and $k$ an integer with $r(k+2) > (r-1)(d+k+1)+1$. 
	Let $C_0, \dots, C_k$ be sets of cardinality at most $2r-1$. Then for any continuous map 
	$f\colon C_0 * \dots * C_k \longrightarrow \R^d$ there are $r$ pairwise disjoint faces $\sigma_1, \dots, \sigma_r$ 
	of $C_0 * \dots * C_k$ such that $f(\sigma_1) \cap \dots \cap f(\sigma_r) \ne \emptyset$.
\end{theorem}

For short and elementary proofs see~\cite{blagojevic2014}. We will obtain a common generalization 
of Theorem~\ref{thm:vanKampen} and Theorem~\ref{thm:colored-tverberg}; see Corollary~\ref{cor:constraints}.

\begin{theorem}[Optimal colored Tverberg theorem, Blagojevi\'c, Matschke, and Ziegler~\cite{blagojevic2009}]
\label{thm:opt-colored}
	Let $r \ge 2$ be a prime, $d \ge 0$ an integer, and $N=(r-1)(d+1)$. 
	Let $C_0, \dots, C_k$ be disjoint sets of cardinality at most $r-1$ such that $|C_0 \cup \dots \cup C_k| = N+1$. 
	Then for any continuous map $f\colon C_0 * \dots * C_k \longrightarrow \R^d$ there are $r$ pairwise 
	disjoint faces $\sigma_1, \dots, \sigma_r$ of $C_0 * \dots * C_k$ such that 
	$f(\sigma_1) \cap \dots \cap f(\sigma_r) \ne \emptyset$.
\end{theorem}

\section{A proof of Kneser's conjecture}
\label{sec:kneser}

\noindent
We recall that $\KG(k,n)$ denotes the graph with vertex set the $k$-element subsets of $\{1, \dots, n\}$
and an edge between two vertices if the corresponding sets are disjoint. We denote the chromatic number
of a graph $G$ by~$\chi(G)$. Kneser's conjecture that $\chi(\KG(n, 2n+k)) = k+2$ was proved by Lov\'asz~\cite{lovasz1978}. 
B\'ar\'any gave a different proof~\cite{barany1978} that also used the Borsuk--Ulam theorem in an essential way and the geometry
of finite point sets on the sphere. Subsequent proofs are due to Greene~\cite{greene2002} and Matou\v sek~\cite{matousek2004}. 
The latter proof is the first purely combinatorial proof of Kneser's conjecture. 

Here we first present a short proof of Kneser's conjecture, which we will extend to the hypergraph setting in
Section~\ref{sec:hypergraphs}. The proof builds on the topological Radon theorem, which is a corollary to the Borsuk--Ulam theorem.

\begin{lemma}
\label{lemma:kneser}
	Let $c\colon V(\KG(n,N+1)) \longrightarrow \{1, \dots, k\}$ be a proper $k$-coloring. 
	Then there is a continuous map $C\colon \Delta_N \longrightarrow \R^k$ with the 
	property that for disjoint faces $\sigma_1$ and $\sigma_2$ of $\Delta_N$ we have 
	that $C(x_1) = C(x_2)$ for $x_1 \in \sigma_1$ and $x_2 \in \sigma_2$ implies 
	$x_1, x_2 \in \Delta_N^{(n-2)}$.
\end{lemma}

\begin{proof}
	Let $G$ be the graph of subsets of $\{1, \dots, N+1\}$ that have cardinality \emph{at least}~$n$
	and an edge between two vertices if the corresponding subsets are disjoint. Then 
	$\KG(n,N+1)$ naturally is a subgraph of~$G$ of the same chromatic number. The coloring $c$
	can be extended to a proper $k$-coloring 
	$c'$ of~$G$: for any subset $\sigma \subseteq \{1, \dots, N+1\}$ of cardinality at least $n$ choose some 
	$\tau \subseteq \sigma$ of cardinality~$n$ and define $c'(\sigma) = c(\tau)$. For example, we could 
	choose $\tau$ such that $c(\tau)$ is minimal among all $n$-subsets of $\sigma$. The map $c'$ maps 
	disjoint sets to distinct values since $c$ does. 
	
	We can now define the map $C$ as an affine map on the barycentric subdivision $\Delta_N'$
	of~$\Delta_N$. The vertices of $\Delta_N'$ correspond to the faces of~$\Delta_N$. Given some
	$\ell$-face $\sigma$ of $\Delta_N$ define $C(\sigma) = e_{c'(\sigma)}$, where $e_1, \dots, e_k$ 
	denotes the standard basis of~$\R^k$, for $\ell \ge n-1$ and $C(\sigma) = 0$ otherwise. Then 
	extend $C$ affinely onto the faces of~$\Delta_N'$.
	
	Let $\sigma$ be a face of~$\Delta_N$ and $x \in \sigma$ some point. Then 
	$C(x) \in \conv(\{e_{c'(\tau)} \: | \: \tau \subseteq \sigma\} \cup \{0\})$.
	Let $\sigma_1$ and $\sigma_2$ be disjoint faces of $\Delta_N$ and $x_1 \in \sigma_1$, 
	$x_2 \in \sigma_2$ with $C(x_1)=C(x_2)$. Since for any subfaces $\tau_1 \subseteq \sigma_1$
	and $\tau_2 \subseteq \sigma_2$ of dimension $\ge n-1$ we have that $c'(\tau_1) \ne c'(\tau_2)$,
	we conclude that $C(x_1) = 0 = C(x_2)$ and thus $x_1, x_2 \in \Delta_N^{(n-2)}$.
\end{proof}

The map $C$ constructed in Lemma~\ref{lemma:kneser} can be used as a constraint function in the 
sense of~\cite{blagojevic2014}. This immediately yields a proof of Kneser's conjecture:

\begin{theorem}[Kneser's conjecture, Lov\'asz~\cite{lovasz1978}]
\label{thm:kneser}
	$\chi(\KG(n, 2n+k)) \ge k+2$
\end{theorem}

\begin{proof}
	Let $N = 2n+k-1$ and let $f \colon \Delta_N \longrightarrow \R^{2n-3}$ be a 
	general position affine map. Suppose there was a proper $(k+1)$-coloring
	$c\colon V(\KG(n, 2n+k)) \longrightarrow \{1, \dots, k+1\}$. Then by Lemma~\ref{lemma:kneser}
	this induces a continuous map $C\colon \Delta_N \longrightarrow \R^{k+1}$ such that
	if $x_1$ and $x_2$ are in disjoint faces of $\Delta_N$ with $C(x_1) = C(x_2)$ then
	$x_1, x_2 \in \Delta_N^{(n-2)}$. 
	
	Consider the continuous map $F \colon \Delta_N \longrightarrow \R^{2n+k-2},
	x \mapsto (f(x), C(x))$. By the topological Radon theorem there are points
	$x_1$ and $x_2$ in disjoint faces of $\Delta_N$ with $F(x_1) = F(x_2)$. Then $C(x_1) = C(x_2)$
	implies that $x_1, x_2 \in \Delta_N^{(n-2)}$. Since also $f(x_1) = f(x_2)$ we get the
	contradiction that two $(n-2)$-faces intersect in $\R^{2n-3}$ for the general 
	position map~$f$.
\end{proof}

A slightly more careful analysis shows that this reasoning actually implies Dol'nikov's generalization~\cite{dolnikov1988}
of Theorem~\ref{thm:kneser}. We do not carry this out here since the generalization to hypergraphs 
presented in the next section will be even more general. See also Section~\ref{sec:schrijver} for
further generalizations.

\section{Generalizations to hypergraph colorings}
\label{sec:hypergraphs}

\noindent
The proof of Theorem~\ref{thm:kneser} easily generalizes to the hypergraph setting, where now instead of
the topological Radon theorem we use the topological Tverberg theorem. A \emph{hypergraph} on \emph{vertex set}
$V$ is a set~$E$ of subsets of~$V$. A \emph{hyperedge} is any element of~$E$. We will restrict our attention to
hypergraphs where the hyperdges have cardinality at least two. A hypergraph is \emph{$r$-uniform} if all hyperedges
have cardinality~$r$. A \emph{$k$-coloring} of a hypergraph
on vertex set $V$ is a function $c\colon V \longrightarrow \{1, \dots, k\}$ such that any hyperdge contains vertices
$v$ and $w$ with $c(v) \ne c(w)$. The least $k$ such that the hypergraph $H$ has a $k$-coloring is its
\emph{chromatic number}~$\chi(H)$. A \emph{partial hypergraph} of a hypergraph $H$ is obtained by
removing hyperedges from~$H$. A \emph{subhypergraph} is obtained by removing vertices.

Let $L$ be a simplicial complex, $K \subseteq L$ a subcomplex, and $r \ge 2$ an integer.
Denote by $\KG^r(K,L)$ the associated \emph{generalized Kneser hypergraph},
that is the $r$-uniform hypergraph with vertices the inclusion-minimal faces in $L$ that are not 
contained in $K$ and a hyperedge of $r$ vertices precisely if the $r$ corresponding
faces are pairwise disjoint. The minimal nonfaces of $K$ are sometimes also called 
\emph{missing faces}.

If $K$ is a skeleton of the simplex, say $K = \Delta_N^{(n-2)}$, then $\KG(K, \Delta_N)$ is the usual Kneser
graph $\KG(n, N+1)$ of $n$-subsets of~$\{1, \dots, N+1\}$. The $r$-uniform hypergraph $\KG^r(K, \Delta_N)$
is called \emph{Kneser hypergraph} and denoted by~$\KG^r(n, N+1)$.

Let $G$ be a system of nonempty subsets of~$\{1, \dots, N+1\}$. Then Kriz~\cite{kriz1992} defines the \emph{$r$-uniform intersection hypergraph}
$[G,r]$ to have vertex set $G$ and a hyperedge $\{M_1, \dots, M_r\}$ for $M_i \in G$ precisely if the $M_i$ are 
pairwise disjoint. Our first goal will be to show that there is no loss of generality in considering the hypergraphs 
$\KG^r(K, \Delta_N)$ compared to~$[G,r]$. Not every hypergraph $[G,r]$ is isomorphic to some $\KG^r(K, \Delta_N)$, but
it dismantles to such a hypergraph, that is, deleting a set from $G$ that is not inclusion-minimal in $G$ has the effect
of deleting a vertex $v$ of $[G,r]$ that is dominated by another vertex $w$ in the sense that replacing $v$ by $w$ in
any hyperedge yields another hyperedge.

\begin{lemma}
	Given a finite system of nonempty sets $G$, let $G'$ consist of those sets in $G$ that are inclusion-minimal in~$G$. 
	Then for any integer $r \ge 2$ we have $\chi([G,r]) = \chi([G',r])$.
\end{lemma}

\begin{proof}
	Since $[G',r]$ is a subhypergraph of $[G,r]$ we have that $\chi([G,r]) \ge \chi([G',r])$. 
	Let $c$ be a proper vertex coloring of~$[G',r]$. Given two sets $\sigma' \subset \sigma$
	with $\sigma' \in G'$ and $\sigma \in G$ and corresponding vertices $v'$ of $\sigma'$ 
	and $v$ of~$\sigma$, let $c(v) = c(v')$. This results in a proper vertex coloring of~$[G,r]$.
	This is due to the fact that for any hyperedge $\{v, w_2, \dots, w_r\}$ of $[G,r]$, the set
	$\{v', w_2, \dots, w_r\}$ is a hyperedge of~$[G',r]$ since $\sigma' \subset \sigma$.
\end{proof}

From now on we will assume that $G$ does not contain two distinct sets $\sigma, \tau \in G$ with $\sigma \subset \tau$
since deleting supersets does not affect the chromatic number of~$[G,r]$.

\begin{lemma}
\label{lem:K}
	Let $G'$ be a system of nonempty subsets of $\{1, \dots, N+1\}$ such that for any two distinct sets
	$\sigma, \tau \in G'$ neither is a subset of the other. 
	Let $K$ be the 
	simplicial complex on vertex set $\{1, \dots, N+1\}$ that contains all sets $\tau \subseteq \{1, \dots, N+1\}$ 
	as faces such that no subset of $\tau$ is in~$G'$. Then $\KG^r(K, \Delta_N) = [G',r]$.
\end{lemma}

\begin{proof}
	Let $\sigma \in G'$. Then all proper subsets of $\sigma$ are not contained in~$G'$. Thus all proper
	subsets of $\sigma$ are faces of~$K$, while $\sigma$ itself is not. This implies that $\sigma$ is a 
	vertex of~$\KG^r(K, \Delta_N)$. Conversely, if $\sigma \subseteq \{1, \dots, N+1\}$ is a minimal 
	nonface of~$K$ then no proper subset of $\sigma$ is in~$G'$. Since $\sigma$ itself is not a face 
	of $K$ it must be in~$G'$. This shows that $\KG^r(K, \Delta_N)$ and $[G',r]$ have the same vertices.
	They coincide as hypergraphs since vertices are connected by hyperedges by the same condition 
	of pairwise disjointness of the corresponding sets for both hypergraphs.
\end{proof}

Kriz defines the \emph{$r$-width} $\omega(G,r)$ of $[G,r]$ as the minimal integer $k$ such that there exist $r$ subsets $M_i$ of
$\{1, \dots, N+1\}$ and such that no subset of any $M_i$ is in $G$ and $|\bigcup_i M_i| = N+1-k$. Kriz's main result then is:

\begin{theorem}[Kriz~{\cite[Theorem~2.4]{kriz1992}}, see also~\cite{kriz2000c}]
\label{thm:kriz}
	$\displaystyle \chi([G,r]) \ge \frac{\omega(G,r)}{r-1}$
\end{theorem}

This contains as a special case the result of Alon, Frankl, and Lov\'asz~\cite{alon1986} on the chromatic number of
Kneser hypergraphs, conjectured by Erd\H os~\cite{erdos1973}. A simpler proof of Theorem~\ref{thm:kriz} using the $\Z/p$-index
is due to Matou\v sek~\cite{matousek2002}. Combinatorial proofs of this result and several of its variants and
extensions are due to Ziegler~\cite{ziegler2002}. See also the erratum~\cite{ziegler2006} and the clarifications
by Lange and Ziegler~\cite{lange2007}.

That no subset of $M_i$ is contained in $G$ is equivalent to $M_i$ spanning a face of the simplicial complex $K$
constructed above. The condition $|\bigcup_i M_i| \le N+1-k$ then is equivalent to $\sum_{i=1}^r \dim \sigma_i \le N+1-k-r$
for any $r$ pairwise disjoint faces $\sigma_1, \dots, \sigma_r$ of~$K$ with $K$ as in Lemma~\ref{lem:K}. 
The inequality of Theorem~\ref{thm:kriz} is purely
combinatorial. We will extend the proof in Section~\ref{sec:kneser} to obtain a lower bound for the chromatic number
of intersection hypergraphs that takes the topology of the simplicial complex $K$ into account. The following lemma
is the $r$-fold analog of Lemma~\ref{lemma:kneser}.

\begin{lemma}
\label{lemma:r-kneser}
	For simplicial complexes $K \subseteq L$ let $c\colon V(\KG^r(K,L)) \longrightarrow \{1, \dots, k\}$ 
	be a proper vertex coloring. Then there is a continuous map $C\colon L \longrightarrow \R^k$
	such that $C(x_1) = \dots = C(x_r)$ for $x_1, \dots, x_r$ in $r$ pairwise disjoint faces of $L$ 
	implies that $x_i \in K$ for all $i = 1, \dots, r$.
\end{lemma}

\begin{proof}
	Denote by $H$ the $r$-uniform hypergraph with vertex set the faces of $L$ that are not contained 
	in~$K$ and a hyperedge for any set of $r$ pairwise disjoint faces. The hypergraph $\KG^r(K,L)$ is
	naturally a subhypergraph of~$H$ and the coloring $c$ can be extended to a proper $k$-coloring
	$c'$ of~$H$: for any face $\sigma$ of $L$ that is not contained in~$K$ let
	$$c'(\sigma) = \min\{c(\tau) \: | \: \tau \subseteq \sigma, \tau \ \text{is a minimal non-face of} \ K\}.$$
	The vertices of the barycentric subdivision $L'$ of~$L$ correspond to the faces of~$L$.
	Thus we can think of $c'$ as a map $c'\colon V(L') \setminus V(K') \longrightarrow \{1, \dots, k\}$.
	Define the affine map $C\colon L' \longrightarrow \R^k$ by setting $C(x) = 0$ for $x \in K$ and 
	$C(x) = e_{c'(x)}$ for any vertex $x$ of $L'$ that is not contained in~$K$, where $(e_1, \dots, e_k)$
	denotes the standard basis of~$\R^k$.
	
	As before for $x \in \sigma$ we have that $C(x) \in \conv(\{e_{c'(\tau)} \: | \: \tau \subseteq \sigma\} \cup \{0\})$.
	Let $A_1, \dots, A_r \subseteq \{1, \dots, k\}$. Then $\bigcap_{i=1}^r \conv(\{e_t \: | \: t \in A_i\} \cup \{0\}) \ne \{0\}$
	if and only if $\bigcap_{i=1}^r A_i \ne \emptyset$. Thus since $c'$ is a proper coloring we have that for
	$r$ pairwise disjoint faces $\sigma_1, \dots, \sigma_r$ of $L$ the intersection $\bigcap_{i=1}^r C(\sigma_i)$
	is equal to~$\{0\}$. This concludes the proof since $C(x) = 0$ implies that $x \in K$.
\end{proof}

We are now ready to prove the main result. Analogous to the case $r=2$ presented in Section~\ref{sec:kneser}
it follows from using the map $C$ of Lemma~\ref{lemma:r-kneser} as a constraint function in the sense of~\cite{blagojevic2014}.

\begin{theorem}
\label{thm:constraints}
	Let $d, k \ge 0$ and $r \ge 2$ be integers, $K \subseteq L$ simplicial complexes such that
	for every continuous map $F\colon L \longrightarrow \R^{d+k}$ there are $r$ pairwise disjoint
	faces $\sigma_1, \dots, \sigma_r$ of $L$ such that $F(\sigma_1) \cap \dots \cap F(\sigma_r) \ne \emptyset$.
	Suppose $\chi(\KG^r(K, L)) \le k$. Then for every continuous map
	$f\colon K \longrightarrow \R^d$ there are $r$ pairwise disjoint faces $\sigma_1, \dots, \sigma_r$
	of $K$ such that $f(\sigma_1) \cap \dots \cap f(\sigma_r) \ne \emptyset$.
\end{theorem}

\begin{proof}
	Extend $f$ to a continuous map $f'\colon L \longrightarrow \R^d$. Let $C\colon L \longrightarrow \R^k$
	be the continuous map given by an proper $k$-coloring of the vertices of $L$ and Lemma~\ref{lemma:r-kneser}.
	Consider the continuous map $F\colon {L \longrightarrow \R^{d+k}}$, ${x \mapsto (f'(x), C(x))}$. 
	Then there are $x_1, \dots, x_r$ in $r$ pairwise disjoint faces of $L$
	with $F(x_1) = \dots = F(x_r)$. The equality $C(x_1) = \dots = C(x_r)$ guarantees that $x_1, \dots, x_r
	\in K$ by Lemma~\ref{lemma:r-kneser}.
\end{proof}

In light of the topological Tverberg theorem, Theorem~\ref{thm:top-tverberg}, the following corollary is immediate.

\begin{corollary}
\label{cor:constraints}
	Let $d, k \ge 0$ be integers, $r \ge 2$ be a prime power, $N \ge (r-1)(d+k+1)$, and $K \subseteq \Delta_N$ 
	with $\chi(\KG^r(K,\Delta_N)) \le k$. Then for every continuous map
	$f\colon K \longrightarrow \R^d$ there are $r$ pairwise disjoint faces $\sigma_1, \dots, \sigma_r$
	of $K$ such that $f(\sigma_1) \cap \dots \cap f(\sigma_r) \ne \emptyset$.
\end{corollary}

This corollary remains true for any integer $r \ge 2$ if the inequality $N \ge (r-1)(d+k+1)$ is replaced by
$N \ge N(r,d+k)$, that is, if the topological Tverberg conjecture holds for parameters $r,d+k$, and~$N$.
Corollary~\ref{cor:constraints} is a common generalization of the generalized van Kampen--Flores theorem,
Theorem~\ref{thm:vanKampen}, and the colored Tverberg theorem, Theorem~\ref{thm:colored-tverberg}:
If $K = \Delta_N^{(k)}$ then for $r(k+2) > N+1$ we have that $\KG^r(K,\Delta_N)$ does not contain any hyperedges
by the pigeonhole principle and thus admits a $1$-coloring. This proves Theorem~\ref{thm:vanKampen}.
If $K = C_1 * \dots *C_k$ is the join of $k$ discrete sets of cardinality at most $2r-1$ then missing faces of $K$ 
are precisely those edges with both endpoints in one~$C_j$. Color such an edge with~$j$. Then at most $r-1$
pairwise disjoint edges can receive color $j$ by the pigeonhole principle. Thus this defines a $k$-coloring of~$\KG^r(K, \Delta_N)$,
and this proves Theorem~\ref{thm:colored-tverberg}.

Whenever we can exhibit a continuous map $f\colon K \longrightarrow \R^d$ that does not identify points
from $r$ pairwise disjoint faces, we get a lower bound for the chromatic number of~$\KG^r(K,\Delta_N)$.
We explicitly formulate this contrapositive as a corollary:

\begin{corollary}
\label{cor:lower-bounds}
	Let $d \ge 0$ be an integer, $r\ge 2$ be a prime power, and $N \ge (r-1)(d+1)$. Let $K \subseteq \Delta_N$ such that 
	there exists a continuous map $f\colon K \longrightarrow \R^d$ with the property that for any $r$ pairwise 
	disjoint faces $\sigma_1, \dots, \sigma_r$ of $K$ we have that $f(\sigma_1) \cap \dots \cap f(\sigma_r) = \emptyset$.
	Then $\chi(\KG^r(K,\Delta_N)) \ge \lfloor\frac{N}{r-1}\rfloor-d$.
\end{corollary}

\begin{proof}
	Let $k \ge 0$ and $0\le m < r-1$ be integers such that $N=(r-1)(d+k+1)+m$. By Corollary~\ref{cor:constraints} we have 
	that $\chi(\KG^r(K,\Delta_N)) \ge k+1$. Now $k+1 = \frac{N-m}{r-1}-d > \frac{N-(r-1)}{r-1}-d \ge \lfloor\frac{N}{r-1}\rfloor-d-1$.
	Thus $k+1 \ge \lfloor\frac{N}{r-1}\rfloor-d$.
\end{proof}

Corollary~\ref{cor:lower-bounds} is more general than Kriz's bound $\chi(\KG^r(K,\Delta_N)) \ge \frac{\omega(K,r)}{r-1}$ if $r-1$ divides~$N$.
By definition of $\omega(K,r)$ we have that $r\dim K \le N+1-\omega(K, r)-r$ and thus $\frac{\omega(K, r)}{r-1} \le \frac{N+1-r\dim K-r}{r-1}$.
Choose the integer $d\ge 0$ such that $(r-1)d = r\dim K +m$ for some $0 < m \le r-1$. Then we have for any (strong) general position 
affine map $f\colon K \longrightarrow \R^d$ that $f(\sigma_1) \cap \dots \cap f(\sigma_r) = \emptyset$ for pairwise disjoint faces 
$\sigma_i$ by a codimension count. (See Perles and Sigron~\cite{perles2014} for the notion of strong general position.)

We now get that $$\frac{\omega(K, r)}{r-1} \le \frac{N+1-r\dim K-r}{r-1} = \frac{N-(r-1)d+m-(r-1)}{r-1} \le \frac{N}{r-1}-d.$$

Hence, if $N$ is divisible by $r-1$ we have that $\frac{\omega(K, r)}{r-1} \le \lfloor\frac{N}{r-1}\rfloor-d$. 

\begin{example}
	This naive general position bound for $\chi(\KG^r(K, \Delta_N))$ can for certain parameters be worse by one 
	compared to Kriz's bound: let $r=3$, $d=1$, $N=5$, and~$K =\Delta_N^{(0)}$. Then $\omega(K,r) = N+1-r\dim K = 3$
	and thus $\chi(\KG^r(K, \Delta_N)) \ge \omega(K,r)/(r-1) = \frac32$. While the bound in Corollary~\ref{cor:lower-bounds}
	works out to be $\chi(\KG^r(K,\Delta_N)) \ge \lfloor\frac{N}{r-1}\rfloor-d = \lfloor\frac{5}{2}\rfloor-1 = 1$.
\end{example}

We recover (or almost recover) Kriz's bound for $\chi(\KG^r(K,\Delta_N))$ if the images of $r$ pairwise disjoint faces of $K$
under a general position affine map $K \longrightarrow \R^d$ do not intersect for codimension reasons. The bound of 
Corollary~\ref{cor:lower-bounds} improves whenever we can produce a continuous map $f\colon K \longrightarrow \R^d$
such that the images of any $r$ pairwise disjoint faces do not intersect, while a codimension count is not sufficient to
ascertain that. For example, the bound of Corollary~\ref{cor:lower-bounds} is better than the bound of Theorem~\ref{thm:kriz}
by an arbitrarily large margin even for $r=2$, by producing simplicial complexes of dimension $d$ that embed into~$\R^d$:

\begin{example} 
	Let $K = \{\{1, 2, \dots, t\}, \{2, 3, \dots, t+1\}, \dots, \{N+1, 1, 2, \dots, t-1\}\}$. Then for $N \ge 2t+1$ odd $K$ is homeomorphic to
	$B^{t-2} \times S^1$ and thus embeds into~$\R^{t-1}$. This follows from extending K\"uhnel's arguments~\cite{kuhnel1986}.
	(For example, one can realize $K$ in the Schlegel diagram of
	a cyclic polytope; see Section~\ref{sec:schrijver} for more general results.) Hence Corollary~\ref{cor:lower-bounds} gives the bound
	$\chi(\KG^2(K,\Delta_N)) \ge N-t+1$, while Theorem~\ref{thm:kriz} only guarantees that 
	$\chi(\KG^2(K,\Delta_N)) \ge N-2t+1$. A greedy coloring shows that the bound $\chi(\KG^2(K,\Delta_N)) \ge N-t+1$ is tight.
\end{example}

\begin{remark}
\label{rem:greedy}
	In fact, a greedy coloring shows that the bound of Corollary~\ref{cor:lower-bounds} is tight whenever $r-1$ divides $N$ and 
	$\dim K \ge (r-1)d$: label the vertices of $K$ by $1, 2, \dots, N+1$ in such a way that $\{{N-(r-1)d+1},$ ${N-(r-1)d+2}, \dots, N+1\}$
	determines a face. Given a missing face $\sigma$ of~$K$ let $k \in \sigma$ be the vertex with the least label. Now color
	$\sigma$ by $\lceil \frac{k}{r-1} \rceil$. This is a proper hypergraph coloring which uses at most $\frac{N}{r-1}-d$ colors.
\end{remark}

Theorem~\ref{thm:kriz} holds for any integer $r \ge2$ while Corollary~\ref{cor:lower-bounds} only holds for prime powers.
In fact, Kriz's proof of Theorem~\ref{thm:kriz} works for primes, and the general case is established by an induction over
the number of prime divisors~\cite{kriz2000c}. Such an induction must fail for the more general setting of Corollary~\ref{cor:lower-bounds}
since it turns out to be wrong for $r$ not a prime power:

\begin{example}
\label{ex:nonprimepower}
	Let $r\ge6$ be an integer that is not a power of a prime, $N = (r-1)(rk+2)$, and let $K = \Delta_N^{((r-1)k)}$.
	Then for $k\ge2$ there is a continuous map $f\colon K \longrightarrow \R^{rk}$ such that for any $r$ pairwise
	disjoint faces $\sigma_1, \dots, \sigma_r$ of $K$ we have that 
	$f(\sigma_1) \cap \dots \cap f(\sigma_r) = \emptyset$; see~\cite{mabillard2016}.
	The minimal nonfaces of $K$ have $(r-1)k+2$ elements and thus $r$ pairwise disjoint minimal nonfaces involve
	$r((r-1)k+2) = (r-1)rk+2r = (r-1)(rk+2)+2 > N+1$ vertices. Thus $\KG^r(K, \Delta_N)$ does not contain any 
	hyperedges, which implies that $\chi(\KG^r(K, \Delta_N)) = 1$, while if Corollary~\ref{cor:lower-bounds} was
	true for $r$ we would obtain the lower bound $\chi(\KG^r(K, \Delta_N)) \ge \lfloor\frac{N}{r-1}\rfloor-rk = 2$.
\end{example} 

\begin{conjecture}
	Corollary~\ref{cor:lower-bounds} remains true for every integer $r\ge2$ if the map $f$ is required to be affine.
\end{conjecture}

We can get a proper extension of Corollary~\ref{cor:constraints} for $r$ a prime by applying the optimal colored 
Tverberg theorem of Blagojevi\'c, Matschke, and Ziegler~\cite{blagojevic2009}, Theorem~\ref{thm:opt-colored}, to Theorem~\ref{thm:constraints}:

\begin{corollary}
\label{cor:optimal-colored}
	Let $d, k \ge 0$ be integers and $r \ge 2$ be a prime. Let $C_1, \dots, C_m$ be sets of at most $r-1$ points with 
	$|\bigcup_{i=1}^m C_i| > (r-1)(d+k+1)$, $L = C_1 * \dots * C_m$, and $K \subseteq L$ 
	with $\chi(\KG^r(K,L)) \le k$. Then for every continuous map
	$f\colon K \longrightarrow \R^d$ there are $r$ pairwise disjoint faces $\sigma_1, \dots, \sigma_r$
	of $K$ such that $f(\sigma_1) \cap \dots \cap f(\sigma_r) \ne \emptyset$.
\end{corollary}

Lastly, as we obtain lower bounds for the chromatic number of $\KG^r(K, \Delta_N)$ by exhibiting a continuous map
$K \longrightarrow \R^d$ without $r$-fold intersections among its pairwise disjoint faces, and the existence of such a
map is equivalent to the existence of a certain $\Sym_r$-equivariant map in the $r$-metastable range due to a recent
result of Mabillard and Wagner~\cite{mabillard2016}, we can obtain a lower bound in these cases by establishing the existence of equivariant
maps. Notice that this is different from the usual approach of proving the nonexistence of a certain equivariant map.

To formulate the result of Mabillard and Wagner we first need some notation. For a simplicial complex $K$ denote by 
$$
K^{r}_{\Delta} = \{(x_1,\dots,x_r) \in \sigma_1 \times \dots \times \sigma_r \: | \: \sigma_i \text{ face of } K, \sigma_i \cap \sigma_j = \emptyset \ \forall i \neq j\}
$$
the \emph{$r$-fold deleted product} of~$K$, which is a polytopal cell complex in a natural way with faces that are products of simplices. 
Denote by $W_r$ the vector space $\{(x_1, \dots, x_r) \in \R^r \: | \: \sum x_i =0\}$ with 
the action by the symmetric group $\Sym_r$ that permutes coordinates. Given a normed vector space $V$, denote the unit sphere in $V$ by~$S(V)$.

\begin{theorem}[Mabillard and Wagner~\cite{mabillard2016}]
\label{thm:metastable}
	Let $d \ge 1$ and $r \ge 2$ be integers, and let $K$ be a finite simplicial complex of dimension $\dim K \le \frac{rd-3}{r+1}$.
	Suppose there is an $\Sym_r$-equivariant map $K^r_\Delta \longrightarrow S(W_r^{\oplus d})$. Then there is a 
	continuous map $f\colon K \longrightarrow \R^d$ such that $f(\sigma_1) \cap \dots \cap f(\sigma_r) \ne \emptyset$
	for all $r$ pairwise disjoint faces $\sigma_1, \dots, \sigma_r$ of~$K$.
\end{theorem}

Combining Theorem~\ref{thm:metastable} and Theorem~\ref{thm:constraints} yields:

\begin{corollary}
\label{cor:metastable}
	Let $d, k \ge 0$ and $r \ge 2$ be integers, $N \ge N(r,d+k)$, and $K \subseteq \Delta_N$ 
	with $\dim K \le \frac{rd-3}{r+1}$ such that an $\Sym_r$-equivariant map $K^r_\Delta \longrightarrow S(W_r^{\oplus d})$
	exists. Then $\chi(\KG^r(K,\Delta_N)) \ge k+1$.
\end{corollary}

In Section~\ref{sec:cyclic} we point out that a potential approach to the notoriously open Conjecture~\ref{conj:kneser-stable} 
is given by Corollary~\ref{cor:metastable} and exhibiting a certain equivariant map.

\section{Vertex-critical subgraphs of Kneser graphs via neighborly polytopes}
\label{sec:schrijver}

\noindent
Schrijver found vertex-critical subgraphs of Kneser graphs~$\KG(n, 2n+k)$, that is, for each choice of parameters
$n \ge 1$ and $k \ge 1$ Schrijver constructs an induced subgraph $G$ of~$\KG(n, 2n+k)$ obtained by deleting vertices such
that $\chi(G) = \chi(\KG(n, 2n+k))$ and deleting any vertex of $G$ decreases the chromatic number~\cite{schrijver1978}.
Schrijver calls a subset of $\{1, \dots, 2n+k\}$ \emph{stable} if it does not contain two successive elements in the
cyclic order of $\{1, \dots, 2n+k\}$, that is, it neither contains $i$ and $i+1$ nor $2n+k$ and~$1$. Define the 
\emph{stable Kneser graph} $\KG'(n, 2n+k)$ to be the subgraph of $\KG(n, 2n+k)$ induced by the vertices 
corresponding to stable sets.

\begin{theorem}[Schrijver~\cite{schrijver1978}]
\label{thm:schrijver}
	We have $\chi(\KG'(n, 2n+k)) = k+2$ and any proper induced subgraph of~$\KG'(n, 2n+k)$ is $(k+1)$-colorable.
\end{theorem}

A combinatorial proof using cyclic matroids is due to Ziegler~\cite{ziegler2002}. Here we will derive Theorem~\ref{thm:schrijver}
as a consequence of Theorem~\ref{thm:constraints} by observing that stable sets are the missing faces of even-dimensional cyclic polytopes.
(The geometry of cyclic polytopes and chromatic numbers of certain graphs were already related to one another in~\cite{adamaszek2014}.) 
The \emph{cyclic $d$-polytope} $C_d(n)$ is the convex hull of $n \ge d+1$ points on the moment curve $\gamma(t) = (t, t^2,\dots, t^d)$.
By Gale's evenness criterion~\cite{gale1963} the facets of a cyclic $2d$-polytope are those $2d$-element subsets $\sigma$ of $\{1, \dots, n\}$
such that between any two elements in $\{1, \dots, n\} \setminus \sigma$ there are an even number of elements in~$\sigma$.
Thus cyclic $2d$-polytopes are \emph{neighborly}, that is, every $d$-element subset determines a face, and $\sigma$
is a $d$-face if and only if it contains one pair $i, i+1 \in \sigma$ or $n,1 \in \sigma$. Moreover neighborly $2d$-polytopes only have
missing faces of dimension~$d$:

\begin{lemma}[Shemer~{\cite[Theorem 2.4]{shemer1982}}]
\label{lem:neighborly}
	Let $P$ be a neighborly $2d$-polytope, and let $\sigma$ be a missing face of~$P$. Then $\dim\sigma=d$.
\end{lemma}

The stable $(d+1)$-element subsets of $\{1, \dots, n\}$ are thus exactly the missing faces of the cyclic $2d$-polytope on $n \ge 2d+2$ vertices. 
Theorem~\ref{thm:schrijver} is now a special case of the following Kneser theorem for missing faces of a sphere:

\begin{theorem}
\label{thm:spherical-kneser}
	Let $K$ be a a triangulation of $S^{d-1}$ on $N+1$ vertices. Then $\chi(\KG^2(K, \Delta_N)) = N+1-d$.
\end{theorem}

\begin{proof}
	The cone $C$ over $K$ is a triangulation of the $d$-ball and hence embeds into~$\R^d$. The complex
	$C$ has the same missing faces as~$K$. Thus $\chi(\KG^2(K, \Delta_N)) = \chi(\KG^2(C, \Delta_N))$.
	Moreover $C$ has $N+2$ vertices, so by Corollary~\ref{cor:lower-bounds} we have that
	$\chi(\KG^2(C, \Delta_N)) \ge N+1-d$. This bound is tight by Remark~\ref{rem:greedy}.
\end{proof}

The stable Kneser graph $\KG'(n, 2n+k)$ is equal to $\KG^2(\partial C_{2n-2}(2n+k), \Delta_{2n+k-1})$.
Applying Theorem~\ref{thm:spherical-kneser} for $N=2n+k-1$ and $d=2n-2$ yields $\chi(\KG'(n, 2n+k)) = k+2$.
More generally, any neighborly 
polytope yields a proper induced subgraph of a Kneser graph of the same chromatic number:

\begin{corollary}
\label{cor:neighborly-kneser}
	Let $P$ be a neighborly $2d$-polytope on $N+1$ vertices distinct from the simplex.  
	Then $\KG^2(\partial P, \Delta_N)$ is a subgraph of the Kneser
	graph~$\KG(d+1,N+1)$ and $$\chi(\KG^2(\partial P, \Delta_N)) = N+1-2d=\chi(\KG(d+1,N+1)).$$
\end{corollary}

This is an immediate consequence of Theorem~\ref{thm:spherical-kneser} and Lemma~\ref{lem:neighborly}.
Since neighborly polytopes are plentiful, see Padrol~\cite{padrol2013}, so are induced subgraphs of $\KG(n, 2n+k)$ with
chromatic number $k+2$ and the same number of vertices as $\KG'(n, 2n+k)$. (They have the same number of
vertices since neighborly $2d$-polytopes with an equal number of vertices have the same number of $d$-faces.)
We leave it as an open problem whether the subgraphs induced by the missing faces of neighborly polytopes 
are always vertex-critical. 

\section{Points in cyclic position and colorings of Kneser hypergraphs}
\label{sec:cyclic}

\noindent
The approach of Section~\ref{sec:schrijver} can be generalized to the hypergraph setting. A subset $\sigma$ of $\{1, \dots, n\}$ is
called \emph{$s$-stable} if any two elements of $\sigma$ are at distance at least $s$ in the cyclic ordering of $\{1, \dots, n\}$.
In particular, $2$-stable sets are the stable sets defined in Section~\ref{sec:schrijver}. Denote by $\KG^r(k, n)_{s-stab}$ the partial
hypergraph of $\KG^r(k, n)$ induced by those vertices corresponding to $s$-stable sets, that is the vertices of $\KG^r(k, n)_{s-stab}$
correspond to $s$-stable $k$-element subsets of~$\{1, \dots, n\}$.

\begin{conjecture}[Ziegler~\cite{ziegler2002}, Alon, Drewnowski, and \L uczak~\cite{alon2009}]
\label{conj:kneser-stable}
	Let $k, r$ and $n$ be positive integers where $n \ge rk$ and $r \ge 2$. Then we have that 
	$\chi(\KG^r(k,n)_{r-stab}) = \lceil \frac{n-r(k-1)}{r-1} \rceil$.
\end{conjecture}

Alon, Drewnowski, and \L uczak~\cite{alon2009} showed that if Conjecture~\ref{conj:kneser-stable} holds for the hypergraphs\linebreak
$\KG^q(k,n)_{q-stab}$ and $\KG^p(k,n)_{p-stab}$ then it also holds for $\KG^r(k,n)_{r-stab}$ with~$r=pq$.
Conjecture~\ref{conj:kneser-stable} holds for $r=2$ by Theorem~\ref{thm:schrijver} and thus it holds if $r$
is equal to any power of two. Moreover it suffices to prove Conjecture~\ref{conj:kneser-stable} for $r$ a prime.

Meunier showed~\cite{meunier2011} that restricting $\KG^r(k,n)$ to those vertices corresponding to almost stable sets does not
decrease the chromatic number. A set $\sigma \subseteq \{1, \dots, n\}$ is called \emph{almost stable} if
for all $i,j\in \sigma$ we have that $|i-j| \ge 2$. More general is the following recent result:

\begin{theorem}[Alishahi and Hajiabolhassan~\cite{alishahi2015}]
\label{thm:kneser-stable}
	Let $k, n$, and $r$ be positive integers with $r$ even or $n$ and $k$ not congruent mod~$r-1$.
	Then $\chi(\KG^r(k,n)_{2-stab}) = \lceil \frac{n-r(k-1)}{r-1} \rceil$.
\end{theorem}

Here we prove that certain small subhypergraphs of $\KG^r(k,n)$ still have chromatic number 
$\lceil \frac{n-r(k-1)}{r-1} \rceil$ at least if $r-1$ divides $n-1$ where $r$ is a prime power. The subsets of $\{1, \dots, n\}$
we restrict to are $(\frac{r(k-3)}{2(k-1)}+1)$-stable on average. We need the following definitions:
A set $\sigma \subseteq \{1, \dots, n\}$ of cardinality~$k$ defines $k$ possible empty gaps $\{1, \dots, n\} \setminus \sigma$
in the cyclic ordering of $\{1, \dots, n\}$. We denote the cardinality of these gaps by $g_1, \dots, g_k$. The average distance
among cyclically successive elements of $\sigma$ is given by
$\frac1k \sum g_i+1$ and is of course equal to the constant~$\frac{n}{k}$. Suppose that $g_k$ is the maximum among $g_1, \dots, g_k$.
We say that $\sigma$ is \emph{$t$-stable on average} if $t \le \frac{1}{k-1}\sum_{i=1}^{k-1} g_i+1$. Any $t$-stable set is
$t$-stable on average. We do not require $t$ to be an integer. 
Denote by $\KG^r(k,n;t)$ the subhypergraph of $\KG^r(k,n)$ induced by those vertices that correspond
to sets that are $t$-stable on average.

The following lemma is simple to prove; see for example Oppermann and Thomas~{\cite[Lemma~2.7]{oppermann2010}}.
As the proof there is only for the even-dimensional case we include a proof below.

\begin{lemma}
\label{lem:intertwined}
	Let $X_1$ and $X_2$ be sets of pairwise distinct points on the moment curve in~$\R^d$ such
	that $\conv X_1 \cap \conv X_2 \ne \emptyset$. Then there are subsets $Y_1 \subseteq X_1$ and $Y_2 \subseteq X_2$
	of cardinality $\lfloor \frac{d}{2}+1 \rfloor$ and $\lceil \frac{d}{2}+1 \rceil$, respectively, 
	such that $\conv Y_1 \cap \conv Y_2 \ne \emptyset$ and such that the vertices of $Y_1$ and $Y_2$ 
	alternate along the moment curve.
\end{lemma}

\begin{proof}
	A finite set of pairwise distinct points on the moment curve is in general position. Find inclusion-minimal
	sets $Y_1 \subseteq X_1$ and $Y_2 \subseteq X_2$ such that $\conv Y_1 \cap \conv Y_2 \ne \emptyset$.
	Then the codimensions of these convex hulls add up to~$d$, or equivalently $|Y_1| + |Y_2| = d+2$.
	Since cyclic polytopes are neighborly and thus convex hulls of at most $\lfloor \frac{d}{2} \rfloor$ points are
	faces on the boundary of $\conv(X_1 \cup X_2)$, we know that $Y_1$ and $Y_2$ have cardinality 
	$\lfloor \frac{d}{2}+1 \rfloor$ and $\lceil \frac{d}{2}+1 \rceil$. Since those points are in general position
	we know that $\conv Y_1 \cap \conv Y_2$ consists of a single point~$x$ which is in the relative interior
	of both convex hulls.
	
	Continuously moving the points in $Y_1 \cup Y_2$ along the moment curve while keeping them distinct 
	continuously varies the intersection point~$x$ that remains in the relative interior of both convex hulls.
	If two points $y,y'$ of $Y_1$, say, were adjacent along the moment curve with no point of $Y_2$ between them,
	we could continuously move $y$ towards~$y'$. By continuity the intersection $\conv Y_1 \cap \conv Y_2$
	is nonempty when $y = y'$, which is in contradiction to these $d+1$ points being in general position.
\end{proof}

We say that a finite point set $X \subseteq \R^d$ is in \emph{strong general position} if for any pairwise disjoint subsets
$X_1, \dots, X_r \subseteq X$ of cardinality at most $d+1$ the codimension of $\aff X_1 \cap \dots \cap \aff X_r$
is the sum of the codimensions of the affine hulls~$\aff X_i$ provided that this sum is at most~$d+1$. Here we define the empty 
set to have codimension~$d+1$. See Perles and Sigron~\cite{perles2014} for this stronger notion of general position and the fact that
points generically are in strong general position. 

\begin{lemma}
\label{lem:avg}
	Let $r\ge 2$, $k\ge 2$, and $d\ge 1$ be integers such that $(r-1)d > r(k-2)$ and let $t < \frac{r-1}{k-1}\lfloor \frac{d}{2} \rfloor +1$.
	Let $X$ be the set system consisting of those subsets of $\{1, \dots, n\}$ that have cardinality~$k$ and are $t$-stable
	on average. Let $K$ be the simplicial complex that contains $\sigma \subseteq \{1, \dots, n\}$ as a face if no subset
	of $\sigma$ is contained in~$X$.
	Then there is an affine map $f \colon K \longrightarrow \R^d$ such that for any $r$
	pairwise disjoint faces $\sigma_1, \dots, \sigma_r$ of~$K$ we have that $f(\sigma_1) \cap \dots \cap f(\sigma_r) 
	= \emptyset$. 
\end{lemma}

\begin{proof}
	We spread the vertices $1, \dots, n$ in cyclic order along the moment curve such that they are
	in strong general position. Suppose there were $r$ pairwise disjoint faces $\sigma_1, \dots, \sigma_r$ of~$K$ such 
	that ${f(\sigma_1) \cap \dots \cap f(\sigma_r) \ne \emptyset}$. W.l.o.g. we can assume that each $\sigma_i$ has
	dimension at most~$d$. Then the map $f$ preserves dimensions: $\dim f(\sigma_i) = \dim \sigma_i$. Since
	the sum of codimensions of the $\sigma_i$ is at most~$d$ and  $r(d-k+2) > d$, we know that at least one 
	$\sigma_j$ has dimension at least~$k-1$ and thus cardinality~$\ell \ge k$.
	
	Now $f(\sigma_j)$ intersects each~$f(\sigma_i)$, and thus there are subsets $\tau_i \subseteq \sigma_i$
	and $\tau_i' \subseteq \sigma_j$ of cardinality $\ge \lfloor \frac{d}{2}+1 \rfloor$ that are intertwined in the sense 
	of Lemma~\ref{lem:intertwined}, that is, their 
	vertices alternate in~$\{1, \dots, n\}$. Thus at least $\lfloor \frac{d}{2} \rfloor$ elements of $\tau_i$ are not contained in
	the largest gap of~$\sigma_j$. The largest gap of $\sigma_j$ has cardinality at most~$n-\ell-(r-1)\lfloor \frac{d}{2} \rfloor$. Since
	the sum of the gap sizes is~$n-\ell$, we deduce that the sum of the cardinalities of all but the largest gap is
	at least~$(r-1)\lfloor \frac{d}{2} \rfloor$. Thus $\sigma_j$ is $t$-stable on average for~${t = \frac{r-1}{k-1}\lfloor \frac{d}{2} \rfloor + 1}$.
	This is a contradiction since some $k$-element subset of $\sigma_j$ must be $t$-stable on average too, but then $\sigma_j$ is not
	a face of~$K$.
\end{proof}

\begin{theorem}
\label{thm:avg}
	Suppose $r\ge 2$ is a prime power and that $r-1$ divides~$n-1$. Let $k \ge 4$ and let $\displaystyle t = \frac{r(k-3)}{2(k-1)}+1$.  
	Then $\chi(\KG^r(k,n;t)) = \lceil \frac{n-r(k-1)}{r-1} \rceil$.
\end{theorem}

\begin{proof}
	Choose the integer $d$ such that $r(k-1)-1 \ge (r-1)d > r(k-2)$. Then we can verify that
	$t = \frac{r(k-3)}{2(k-1)}+1 =  \frac{r(k-2)-r}{2(k-1)}+1 < \frac{r-1}{k-1}\frac{d-1}{2}+1 \le \frac{r-1}{k-1}\lfloor \frac{d}{2} \rfloor +1$.
	Thus with the simplicial complex $K$ chosen as in Lemma~\ref{lem:avg} we have that there is an affine map 
	$f \colon K \longrightarrow \R^d$ such that for any $r$ pairwise disjoint faces $\sigma_1, \dots, \sigma_r$ of~$K$ we have 
	that $f(\sigma_1) \cap \dots \cap f(\sigma_r) = \emptyset$. Moreover the $k$-element sets that are $t$-stable on average are the
	missing faces of~$K$ by Lemma~\ref{lem:K}: $\KG^r(K, \Delta_{n-1}) = \KG^r(k,n;t)$. By Corollary~\ref{cor:lower-bounds} we get the lower
	bound $\chi(\KG^r(k,n;t)) \ge \frac{n-1}{r-1} - d \ge \frac{n-r(k-1)}{r-1}$. Then $\chi(\KG^r(k,n;t)) = \lceil \frac{n-r(k-1)}{r-1} \rceil$
	since $\frac{n-1}{r-1} - d$ is an integer.
\end{proof}

One could hope for a proof of Conjecture~\ref{conj:kneser-stable} in this way, at least if $r-1$ divides~${n-1}$. 
Let $X$ be the set system consisting of $k$-element subsets of $\{1, \dots, n\}$ that are $r$-stable. 
Let $K$ be the simplicial complex that contains $\sigma \subseteq \{1, \dots, n\}$ as a face if no subset
of $\sigma$ is contained in~$X$. Let $d \le \frac{r}{r-1}(k-1)-\frac{1}{r-1}$. Find a continuous map
$f \colon K \longrightarrow \R^d$ such that for any $r$ pairwise disjoint faces $\sigma_1, \dots, \sigma_r$ 
of~$K$ we have that $f(\sigma_1) \cap \dots \cap f(\sigma_r) = \emptyset$.
By work of Mabillard and Wagner~\cite{mabillard2016}, see Theorem~\ref{thm:metastable} and Corollary~\ref{cor:metastable},
it is sufficient to exhibit an $\Sym_r$-equivariant map $K^r_\Delta \longrightarrow S(W_r^{\oplus d})$ provided that
$(r+1)(k-1) \le rd-3$.

\bibliographystyle{amsplain}


\end{document}